\documentclass[12pt,a4paper]{amsart}
\usepackage{eucal}
\usepackage{enumerate}
\usepackage{fontenc}
\usepackage{amsfonts,amsmath,amssymb,amsthm}
\usepackage{dsfont}
\usepackage{latexsym}
\usepackage{mathbbol}
\usepackage{enumerate}
\usepackage{nomencl}
\usepackage{graphicx}
\usepackage{graphics}
\usepackage{epsfig}
\usepackage{verbatim}
\usepackage{subfigure}
\usepackage{color}
\usepackage{amsthm,amssymb,amsmath,amscd,graphicx,epsfig,hhline,mathrsfs,latexsym}
\usepackage[T1]{fontenc}
\usepackage{yfonts}
\usepackage{amsmath,amsfonts,amstext}
\usepackage{amssymb,amsmath,amscd,graphicx,epsfig,hhline,mathrsfs,yhmath,latexsym,url}
\usepackage[neverdecrease]{paralist}
\usepackage{yfonts,dsfont,microtype}
\usepackage[numbers,sort&compress]{natbib}
\usepackage{enumerate}
\usepackage{mdwlist}
\usepackage{mathrsfs} 
\usepackage[arrow, matrix, curve]{xy}
\usepackage[british]{babel}

\usepackage{paralist}               	
\usepackage[parfill]{parskip}			

\usepackage{amsmath,amsthm}
\usepackage{amssymb}
\newcommand{\mk}{\mathfrak}
\newcommand{\mc}{\mathcal}

\newcommand{\mf}{\mathbf}
\newcommand{\mb}{\mathbb}
\newcommand{\mr}{\mathrm}

\newtheorem*{T*}{Theorem}
\newtheorem*{A*}{Proposition}
\newtheorem*{Cor*}{Corollary}
\newtheorem{T}{Theorem}[section]
\newtheorem{Le}[T]{Lemma}
\newtheorem{R}[T]{Remark}
\newtheorem{Cor}[T]{Corollary}
\newtheorem{A}[T]{Proposition}
\theoremstyle{definition}\newtheorem{D}[T]{Definition}
\newtheorem*{D*}{Definition}
\newtheorem{Ex}[T]{Example}
\def \ol {\overline}
\newcommand{\lin}{\text{lin}}
\newcommand{\en}{\enspace}
\def \ph {\varphi}

\def \L {\mathscr{L}}

\def \S {\mathscr{S}}

\def \ol {\overline}

\def \ph {\varphi}

\def \lin {\operatorname{lin}}

\title[Almost weak polynomial stability]{Almost weak polynomial stability of operators
}

\author[D. Kunszenti-Kov\'{a}cs]{D\'{a}vid Kunszenti-Kov\'{a}cs}
\address{Eotv\"os Lor\'and University, Institute of Mathematics and Numerical Analysis
and Large Networks Research Group, Hungarian Academy of Sciences, 1117 Budapest, P\'{a}zm\'{a}ny P. s\'{e}t\'{a}ny 1/C, Hungary.}
\email{daku@fa.uni-tuebingen.de}

\thanks{The author was supported by ERC Grant No. 227701}
\date\today
\begin{document}

\begin{abstract}
We investigate whether almost weak stability of an operator $T$ on a Banach space $X$ implies its almost weak polynomial stability. We show, using a modified version of the van der Corput Lemma that if $X$ is a Hilbert space and $T$ a contraction, then the implication holds. On the other hand, based on a TDS arising from a two dimensional ODE, we give an explicit example of a contraction on a $C_0$ space that is almost weakly stable, but its appropriate polynomial powers fail to converge weakly to zero along a subsequence of density $1$.
Finally we provide an application to convergence of polynomial multiple ergodic averages.
\end{abstract}

\maketitle

\section{Introduction}

Considering convergence along subsequences is very common in the ergodic theoretical setting (see e.g. \cite{BLRT} and references therein). Polynomial subsequences are of special interest since they arise in a natural way through group extensions (cf. Furstenberg \cite[Chapter 3]{Furst}).
In this paper we study the asymptotics of polynomial subsequences of orbits of contractions on Hilbert spaces, and its consequences for various ergodic theorems.
To avoid ambiguity, we write $\mb{N}_0$ for the set of nonnegative integers, and $\mb{N}^+$ for the set of positive integers.
We denote by $\mathcal{P}\subset\mb{Z}\left[X\right]$ 
the set of all polynomials mapping $\mb{N}^+$ to $\mb{N}_0$,
and by $\mathcal{P}_0\subset\mb{Z}\left[X\right]$ 
the set of all polynomials mapping $\mb{N}_0$ to $\mb{N}_0$ with $p(0)=0$. Further $\Gamma\subset\mb{C}$ denotes the unit circle.

We shall need the following notion to understand the convergence types used in this paper.

\begin{D*}
The \emph{density} of a monotone sequence $\{n_k\}_{k\in\mb{N}^+}\subset\mb{N}^+$ is 
\[
\lim_{n\to\infty}\frac{\left|\left\{k\in \mb{N}^+\left|n_k\leq n\right.\right\}\right|}{n},
\]
whenever the above limit exists.
\end{D*}

With the help of the above definition, we can define the notion of almost weak stability (cf. weak mixing in Zsid\'o \cite{ZS}).

\begin{D*}
A sequence $\{x_n\}_{n\in\mb{N}^+}$ in a Banach space $X$ is called \emph{almost weakly stable} if there exists a sequence $\{n_k\}_{k\in\mb{N}^+}$ with density 1 such that
\[
\mr{weak}\lim_{k\to\infty} x_{n_k}=0.
\]
Let $T$ be a bounded operator on $X$. A vector $x\in X$ is then called \emph{almost weakly stable with respect to} $T$ if its orbit $\{T^n x\}_{n\in\mb{N}^+}$ is almost weakly stable.
Finally the operator $T$ itself is called \emph{almost weakly stable} if every vector $x\in X$ is almost weakly stable with respect to $T$.
\end{D*}

By comparison recall that \emph{weakly stability} requires weak convergence along the whole sequence, not only along one with density $1$, i.e., we have the following definition.

\begin{D*} A sequence $\{x_n\}_{n\in\mb{N}^+}$ in a Banach space $X$ is called \emph{weakly stable} if
\[
\mr{weak}\lim_{n\to\infty} x_{n}=0.
\]
Correspondingly, if $T$ is a bounded operator on $X$, then a vector $x\in X$ is called \emph{weakly stable with respect to} $T$ if its orbit $\{T^n x\}_{n\in\mb{N}^+}$ is weakly stable, while the operator $T$ itself is called \emph{weakly stable} if every vector $x\in X$ is weakly stable with respect to $T$.
\end{D*}

We first take a look at splitting theorems on Hilbert spaces, and show that almost weak stability for contractions also implies almost weak stability of polynomial subsequences of orbits.

The main result of this paper is the following theorem.

\begin{T}\label{T:uaws_pol}
Let $T$ be an almost weakly stable contraction on a Hilbert space $H$. Then $T$ is almost weakly polynomial stable, i.e., for any $h\in H$ and non-constant polynomial $p\in\mc{P}$ the sequence $\{T^{p(j)}h\}_{j=1}^\infty$ is almost weakly stable.
\end{T}

 This is not true for general contractions, as shown by an example.
Thereafter we apply the obtained results to the setting of entangled and multiple polynomial ergodic averages.

%%%%%%%%%%%%%%%%
%
%
%
%
%%%%%%%%%%%%%%%%

\section{Almost weak polynomial stability on Hilbert spaces}\label{Sect:awps}

We start with the classical splitting theorem due to K. Jacobs (see \cite{jacobs57}) characterising the orthogonal complement of the subspace of almost periodic vectors of a semigroup of contractions.

\begin{T}\label{T:spl_jacobs}
Let $H$ be a Hilbert space and $\S\subset \L(H)$ a semigroup of contractions, and let $\ol{\S}$ denote its closure in the weak operator topology.
Then $H$ can be decomposed into $\S$-invariant subspaces as
\begin{equation*}
H = H _{ \mathrm{r} } \oplus  H _{ \mathrm{s} },
\end{equation*}
where $H _{ \mathrm{r} }$ is the space of $\ol{\S}$-reversible elements, i.e.,
\begin{equation*}
H _{ \mathrm{r} }  =
	\{x\in H \colon \forall S\in\ol{\S} \,\exists\, T\in\ol{\S} \text{ such that } TSx=x\}
\end{equation*}
and $H _{ \mathrm{s} }$ consists of the orbits for which $0$ is a weak accumulation point,
\begin{equation*}
H _{ \mathrm{s} }  =
	 \left\{ x \in H \colon 0 \in
	 		 \overline{\{ S x \colon S \in   \S \}}  ^ {\sigma ( H , H ^{ * } ) }  \right\}.
\end{equation*}
\end{T}
These subspaces are usually also referred to as the \textit{reversible} and \textit{stable} subspaces of the dynamical system corresponding to the semigroup $\S$. This theorem has later been generalised quite extensively, and stronger characterizations of both the reversible and the stable part have been obtained in the monothetic case, i.e. when $\S$ is generated by a single operator.
The following proposition is a special case of, e.g., Krengel \cite[Section 2.2.4]{Krengel} or Eisner \cite[Theorem II.4.8]{eisner-book}.

\begin{A}\label{prop:JGdL_Hilbert}
Let $T$ be a contraction on a Hilbert space $H$, and consider the semigroup $\{T^n|n\in\mb{N}^+\}$. Then the above spaces can be characterised as
\[
H_{\mr{s}}=\left\{g\in H\left|\lim_{j\to\infty} T^{n_j}g=0 \mbox{ weakly for some sequence }\{n_j\}_{j=1}^\infty\mbox{ with density }1\right.\right\},
\]
and
\[
H_{\mr{r}}=\overline\lin\left\{h\in H \left|\exists \lambda\in\Gamma\mbox{ such that } Th=\lambda h\right.\right\}.
\]
\end{A}

A second decomposition theorem was first proven by B. Sz\H{o}kefalvi-Nagy and C. Foia\c{s} \cite{Nagy} and H. Langer \cite{Langer} independently,  and then extended by S. Foguel \cite[Theorem 1.1]{Foguel}.

\begin{A}\label{T:spl_unit}
Let $T$ be a contraction on a Hilbert space $H$. Then $H$ has a unique orthogonal decomposition $H=H_u\oplus H_0$ into $T$-invariant subspaces such that $T$ acts as a unitary operator on $H_u$, and its restriction to $H_0$ is completely non-unitary. In addition, these two subspaces satisfy
\[
H_u=\left\{h\in H \left| \|h\|=\|T^nh\|=\|T^{*n}h\|\,\forall n\in\mb{N}^+\right.\right\}
\]
and
\[
\mr{weak}\lim_{n\to\infty}T^ng=\mr{weak}\lim_{n\to\infty}T^{*n}g=0
\]
for each $g\in H_0$.
\end{A}

Combining both results we obtain the following splitting theorem for Hilbert space contractions.

\begin{Cor}\label{T:splitting_auld}
Let $T$ be a contraction on a Hilbert space $H$.
Then there is a unique orthogonal decomposition $H=H_\mr{r}\oplus H_{\mr{us}}\oplus H_{0}$ into three $T$-invariant subspaces such that
\begin{itemize}
\item $H _{ \mathrm{r} }  =
	\overline\lin\left\{h\in H \left|\exists \lambda\in\Gamma\mbox{ such that } Th=\lambda h\right.\right\}$,
\item $T|_{H_{\mr{us}}}$ is unitary and 
 each $g\in H_{\mr{us}}$ is almost weakly stable,
\item $T|_{H_0}$ is completely non-unitary, and $T$ and $T^*$ are both weakly stable on $H_0$.
\end{itemize}
\end{Cor}

Our aim is now to strengthen the characterization of the unitary almost weakly stable part by investigating weak convergence along polynomial sequences. To do so, we introduce the following notion.
\begin{D}
A sequence $\{x_n\}_{n=1}^\infty\subset X$ is \emph{almost weakly polynomial stable} if for any non-constant polynomial $p\in\mc{P}$ there exists a sequence $\{n_j\}_{j=1}^\infty\subset\mb{N}$ with density 1 such that
\[
\mr{weak} \lim_{j\to\infty} x_{p(n_j)}=0.
\]
\end{D}

Van der Corput type inequalities have been used in the context of weakly mixing dynamical systems as a key tool in inductive proofs. They are however also useful when wanting to pass from asymptotics along linear sequences to polynomial sequences. The following Lemma (cf. first statement of Niculescu, Str\"oh, Zsid\'o \cite[Thm. 7.1]{niculescu/stroh/zsido} for an even more general statement) is a stronger version of the one used by Bergelson \cite[Theorem 1.5]{Bergelson1987}, since the assumptions on the sequence $\{h_n\}_{n=1}^\infty$ are weaker.

\begin{Le}[van der Corput]\label{vdC}

 Let $\{h_n\}_{n=1}^\infty$ be a sequence in a Hilbert space $H$ with $\|h_n\|\leq1$. For $j\in\mb{N}^+$ let further
\[
 \gamma_j:=\limsup_{N\to\infty}\left|\frac{1}{N}\sum_{n=1}^N\langle h_n,h_{n+j}\rangle \right|.
\]
Then $\lim_{N\to\infty}\frac{1}{N}\sum_{n=1}^{N}\gamma_n=0$ implies $\lim_{N\to\infty}\frac{1}{N}\sum_{n=1}^{N}h_n=0$.
\end{Le}

This lemma yields norm stability of Ces\`aro means, whilst the characterizations in Theorem \ref{T:splitting_auld} are related to weak convergence. We therefore turn to another variant of the van der Corput lemma, using a stronger assumption to obtain almost weak stability.

The key observation is the following result linking Ces\`aro convergence to $0$ of a positive sequence in $\mb{R}$ to almost weak convergence to $0$ of the same sequence.
\begin{Le}[Koopman--von Neumann]\label{lemma:K-vN}
For a bounded sequence $\{y_n\}_{n=1}^\infty\subset[0,\infty)$ the following assertions are equivalent.
\begin{enumerate}[(a)]
\item $\displaystyle\lim_{n\to\infty}\frac{1}{n}\sum_{k=1}^n y_k=0$.
\item There exists a subsequence $\{n_j\}_{j=1}^\infty$ of $\mb{N}$ with density $1$ such that $\lim_{j\to\infty}y_{n_j}=0$.
\end{enumerate}
\end{Le}
\noindent We refer to e.g.~Petersen \cite[p.~65]{petersen:1983} for
the proof.

 Note that the following version of the van der Corput lemma is similar to the one used by Furstenberg (\cite[Lemma 4.9]{Furst}), but the condition on the sequence $\{h_n\}_{n=1}^\infty$ is here again weaker. This lemma follows from the second statement of Theorem $7.1$ of Niculescu, Str\"oh, Zsid\'o \cite{niculescu/stroh/zsido}, where the proof may also be found.

\begin{Le}[van der Corput for almost weak stability]\label{Le:weak}
 Let $\{h_n\}_{n=1}^\infty$ be a sequence in a Hilbert space $H$ with $\|h_n\|\leq1$ for all $n\in\mb{N}^+$. For each $j\in\mb{N}^+$ let further
\[
 \widetilde{\gamma}_j:=\limsup_{N\to\infty}\frac{1}{N}\sum_{n=1}^N\left|\langle h_n,h_{n+j}\rangle \right|.
\]
Then $\lim_{N\to\infty}\frac{1}{N}\sum_{n=1}^{N}\widetilde{\gamma}_n=0$ implies that $\{h_n\}_{n=1}^\infty$ is almost weakly stable.
\end{Le}

\begin{R}
Note that although the assumptions in Lemma \ref{Le:weak} imply the ones in Lemma \ref{vdC}, there is no direct implication between their conclusions, as norm convergence of Ces\`aro means and almost weak stability are two independent properties.
\end{R}

We can now prove the following characterization of almost weakly stable operators on Hilbert spaces.
The idea is to obtain results for a given sequence by passing to the difference sequence and applying an induction argument, starting from the linear case.

\begin{A*}[Theorem \ref{T:uaws_pol}]
Let $T$ be an almost weakly stable contraction on a Hilbert space $H$. Then $T$ is almost weakly polynomial stable, i.e., for any $h\in H$ and non-constant polynomial $p\in\mc{P}$ the sequence $\{T^{p(j)}h\}_{j=1}^\infty$ is almost weakly stable.
\end{A*}
\begin{proof}
By Theorem \ref{T:spl_unit}, $H$ can be split into an orthogonal sum $H_u\oplus H_0$ of $T$-invariant subspaces such that $T|_{H_u}$ is unitary, whilst $T|_{H_0}$ is weakly stable. The latter part of $T$ is then a fortiori almost weakly polynomial stable. Thus it remains to be shown that this also holds for the unitary part of $T$.
Let therefore  $T$ be an almost weakly stable unitary operator, and take $h\in H$. We shall proceed by induction on the degree of the polynomial $p$.

If $\deg p=1$, then $p$ is of the form $aX+b$. The affine sequence $(an+b)_{n\in\mb{N}^+}$  in $\mb{N}^+$ has positive density $1/a$, and thus for any almost weakly stable operator $T$ the sequence $\left(T^{an+b}h\right)_{n\in\mb{N}^+}$ is almost weakly stable.
Suppose now that for each polynomial $q\in\mc{P}$ with $1\leq\deg q\leq d$ the sequence $\left(T^{q(n)}h\right)_{n\in\mb{N}^+}$ is almost weakly stable. Take $p\in\mc{P}$ with degree $d+1$.
Since $p\in\mc{P}$ is non-constant, there exists $n_0\in\mb{N}^+$ such that $p$ is strictly monotone increasing on $[n_0,\infty)$. Consider the sequence $(h_k)_{k\in\mb{N}^+}\subset H$ defined by
\[
h_k:= T^{p(n_0+k)}h.
\]
Then $\langle h_j,h_{j+n}\rangle=\langle T^{p(n_0+j)}h,T^{p(n_0+j+n)}h\rangle=\langle h,T^{p(n_0+j+n)-p(n_0+j)}h\rangle$.
Now the polynomial $\overline{p}_n\in\mc{P}$ defined by
\[
\overline{p}_n(X):=p(n_0+X+n)-p(n_0+X)
\]
has degree $\deg p-1=d$, hence the sequence $\left(T^{\overline{p}_n(j)}h\right)_{n\in\mb{N}^+}$ is almost weakly stable. But this is by Lemma \ref{lemma:K-vN} equivalent to
\[
\lim_{N\to\infty}\frac{1}{N}\sum_{j=1}^N\left|\langle g,T^{\overline{p}_n(j)}h\rangle\right|=0 \mbox{ for all } g\in H.
\]
Applying this to the case $g=h$ and writing
\[
\widetilde{\gamma}_n:=\limsup_{N\to\infty}\frac{1}{N}\sum_{j=1}^N\left|\langle h_j,h_{j+n}\rangle\right|
\]
 we thus obtain $\widetilde{\gamma}_n=0$ for all $n\in\mb{N}^+$. By Lemma \ref{Le:weak} the sequence $(h_k)_{k\in\mb{N}^+}$ is then almost weakly stable. Since adding finitely many elements to a sequence does not influence its almost weak stability, $\left(T^{p(n)}h\right)_{n\in\mb{N}^+}$ is itself also almost weakly stable.

\end{proof}

\begin{R}\label{R:separable}
By a diagonal argument, it can be shown that if $H$ is separable and $T$ is almost weakly polynomial stable, then there exists a sequence $\{n_j\}_{j\in\mb{N}^+}$ of density $1$ such that $\lim_{j\to\infty} T^{p(n_j)}h=0$ weakly for every $h\in H$, i.e. the polynomial powers of $T$ themselves converge to zero in the weak operator topology along a sequence of density 1.\\
Essentially, for each $g_n$ of the countable separating set in $H$ one passes to a further subsequence of density $1$ such that we have weak convergence along it for each $g_k$, $k\leq n$.
The technical difficulty here is to ensure that after thinning out the original sequence a countable number of times, we still end up with a sequence of the required density $1$.
 For more details on how to obtain such an appropriate subsequence, we refer to Petersen \cite[Remark 2.6.3]{petersen:1983} or Niculescu, Str\"oh, Zsid\'o \cite[Lemma 9.1]{niculescu/stroh/zsido}.
\end{R}

In the following we give an example of an almost weakly stable contraction on a Banach space that is not almost weakly polynomial stable. Thus Theorem \ref{T:uaws_pol} cannot be generalised to arbitrary Banach spaces.

\begin{Ex}\label{Ex:C_0}

This example is based on Example $4.3$ in \cite{EFNS}.
We shall first define a continuous flow $\ph$ on $\Gamma$ and a single curve $\gamma$ in the interior of the unit disk $\mb{D}$.

Let $1$ be a fixed point of the flow, and let the flow on $\Gamma\backslash\{1\}$ be given as the homoclinic orbit of $-1$ in the following way.

\[
\ph_t(-1):=
\left
\{\begin{array}{lcr}
e^{\frac{\pi}{t+1} i }&\mbox{ if }& t\geq0\\
e^{\frac{\pi}{t-1} i}&\mbox{ if }&t\leq0.
\end{array}
\right.
\]

On the curve in the interior of $\mb{D}$, the flow is given by the parametrization of the curve, i.e. $\ph_t(\gamma(s)):=\gamma(s+t)$ for all $s,t\in\mb{R}$.
Let the curve $\gamma(t):=r(t)e^{\omega(t) i}$ be given by

\[
r(t):=\left\{
\begin{array}{lcr}
1-\frac{1}{2t}&\mbox{ if }&t\geq1\\
\frac{e^{t-1}}{2}&\mbox{ if }&t\leq1
\end{array}\right.
\]
and 
\[
\omega(t):=\left\{
\begin{array}{lcr}
-2k\pi-\frac{\pi}{2k^2+2-t}&\mbox{ if } &2k^2-k+2\leq t\leq 2k^2+1\\
-(2k+2)\pi+\frac{\pi}{t-2k^2}&\mbox{ if }&2k^2\leq t\leq 2k^2+k+1\\
-2k\pi+\frac{2k^2-2k+2-t}{k^2}\pi&\mbox{ if }&2k^2-3k+2\leq t\leq 2k^2-k+2\\
-4\pi+\frac{\pi}{t-2}&\mbox{ if }&3\leq t\leq4\\
-t\pi&\mbox{ if }&t\leq3
\end{array}\right.,
\]
where $k$ denotes an arbitrary integer.
Note that the curve is actually obtained as follows. For $t\leq1$, it spirals outwards from $0$ with constant angular speed. From $t=1$ onwards, on its $k$-th round around $0$, it follows radially the same angular speed as the homoclinic orbit $\Gamma\backslash\{1\}$ for angles outside of $(-\pi/k,\pi/k) \mod 2\pi$, and constant angular speed $\pi/k^2$ for angles within that interval. Therefore the flow on $\mf{S}:=\Gamma\cup\{\gamma(t)|t\in\mb{R}\}$ is continuous.

The flow can then be continuously extended to the whole of $\mb{D}\backslash\{0\}$ by piecewise linearization along rays starting at $0$, and continuous extension to $\mb{C}\backslash\mb{D}$ is also easily feasible. Thus the conditions of Example $4.3$ in \cite{EFNS} are fulfilled, and we can apply the results obtained therein.

The induced semigroup $(T(t))_{t\geq0}$ on $C(\mf{S})$ defined by
\[
(T(t)f)(x):=f(\ph_t(x)),\en f\in C(\mf{S}),\,x\in\mf{S}
\]
is then strongly continuous, isometric and weakly relatively compact. Let $(T_0(t))_{t\geq0}$ be its restriction to $C_0(\mf{S}\backslash\{1\})\cong\left\{f\in C(\mf{S})\left|f(1)=0\right.\right\}$.
Consider the discrete semigroup generated by $T_0(1)$. Since $T_0(1)$ has no unimodular eigenvalues, this semigroup is almost weakly stable by Theorem $II.4.1$ in \cite{eisner-book}.
But it can be checked that $\omega(2n^2-4n+3)=-(2n-1)\pi$ for positive integer values of $n$, hence $\lim_{n\to\infty} \gamma(2n^2-4n+3)=-1$. This implies that 
\[
\lim_{n\to\infty}
\langle T_0(1)^{2n^2-4n+3}f, \delta_{\gamma(0)}\rangle=
\lim_{n\to\infty} f(\gamma(2n^2-4n+3))=f(-1),
\]
and so the semigroup does not converge weakly to zero along the polynomial $p(X)=2X^2-4X+3$. Thus the operator $T_0(1)$ is almost weakly, but not almost weakly polynomial stable.
\end{Ex}

%%%%%%%%%%%%%%%%%%%%%%%%%%%
%
%
%
%
%
%%%%%%%%%%%%%%%%%%%%%%%%%%%

\section{Polynomial multiple ergodic averages}

In this section we apply the previous results to the setting of polynomial multiple ergodic averages, much in the vein of Eisner, Kunszenti-Kov\'acs \cite{EKK}.

We first introduce what we mean by a non-commutative dynamical system
and recall two convergence notions on von Neumann algebras. 

\begin{D}
A von Neumann (or non-commutative) dynamical system is a triple
$(\mk{A}, \ph, \beta)$, where $\mk{A}$ is a von Neumann algebra,
$\ph:\mk{A}\to \mb{C}$ is a faithful normal trace, and    
$\beta:\mk{A}\to \mk{A}$ is a $\ph$-preserving $*$-automorphism.
We say that a sequence $(b_n)_{n\in\mb{N}^+}$ in $\mk{A}$ converges \emph{strongly} if it converges in the 
$\ph$-norm $\|b\|_\ph:=\sqrt{\ph(b^*b)}$. It is said to be \emph{weakly} convergent if
\[
\ph(a_0b_n)
\]
converges as $N\to\infty$ for every $a_0\in \mk{A}$. 
\end{D}

Non-commutative dynamical systems and and their convergence properties have received much attention and were studied amongst others by Niculescu, Str\"oh and Zsid\'o \cite{niculescu/stroh/zsido}, Duvenhage \cite{duvenhage}, Beyers, Duvenhage and Str\"oh \cite{beyers/duvenhage/stroh}, Fidaleo \cite{fidaleo:2009}, and Austin, Eisner, Tao \cite{TTT}.

The last mentioned work, \cite{TTT}, studied the question of convergence of the multiple ergodic averages
\[
\frac{1}{N}\sum_{n=1}^N \beta^{n}(a_1)\beta^{2n}(a_2)\cdots\beta^{kn}(a_k)
\]
depending on $k\in\mb{N}^+$, showing that in contrast to the commutative case, one cannot expect convergence in general if $k\geq 3$ . On the other hand it is shown in Section 4 of \cite{EKK} that for every von Neumann dynamical system there is a large class (see below) $\mc{K}$ depending on the system such that the multiple ergodic averages converge strongly whenever
$a_1,\ldots,a_k\in \mc{K}$.
We wish to extend the latter result to multiple averages involving polynomial powers of the *-automorphism $\beta$.

More precisely, let $r,k\in\mb{N}^+$ with $r\leq k$, $\alpha:\{1,\ldots,k\}\rightarrow\{1,\ldots,r\}$ be a surjective mapping and $p_1,p_2,\ldots,p_r\in\mc{P}$. We shall be interested in the convergence of the expression
\begin{equation}\label{eq:mult-erg-ave}
\frac{1}{N^r} \sum_{n_1,\ldots,n_r=1}^N \beta^{s_1}(a_1)\beta^{s_2}(a_2)\cdots\beta^{s_k}(a_k)
\end{equation}
where $s_l:=\sum_{d=1}^l p_{\alpha(d)}(n_{\alpha(d)})$ for each $1\leq l\leq k$. Note that with the choice of $r=1$ and $p_1(n):=n$ we obtain the above mentioned linear case studied in \cite{TTT}.

We recall that by the Gel'fand--Neumark--Segal theory, $\mk{A}$ can be
identified with a dense subspace of a Hilbert space, where the Hilbert
space can be obtained as the completion of $\mk{A}$ with respect to the
$\ph$-norm. Thus, identifying elements of $\mk{A}$ with elements in $H$
and by the standard density argument, strong convergence of the multiple ergodic averages
(\ref{eq:mult-erg-ave}) corresponds to norm convergence in $H$ and
weak convergence of (\ref{eq:mult-erg-ave}) corresponds
to weak convergence in $H$. 

Recall further that for the automorphism
$\beta$ there exists a unitary operator $u\in\mc{L}(H)$ such that
$\beta(a)=uau^{-1}$, see e.g. \cite[Prop. 4.5.3]{kadison/ringrose}. Note that $u$ does not
necessarily belong to $\mk{A}$, and in this context the class $\mc{K}$ mentioned above can be chosen as the subspace of all elements $a\in \mk{A}$ such that $\{a u^n: n\in \mb{N}_0\}$ is relatively compact in $\mc{L}(H)$ for the strong operator topology. This class $\mc{K}$ then in particular contains all compact operators in $\mk{A}$.

Thus, averages (\ref{eq:mult-erg-ave}) take the form 
\begin{equation}\label{eq:ent-ave-for-mult}
\frac{1}{N^r}\sum_{n_1,\ldots,n_r=1}^N u^{p_{\alpha(1)}(n_{\alpha(1)})}a_1u^{p_{\alpha(2)}(n_{\alpha(2)})}a_2\cdots u^{p_{\alpha(k)}(n_{\alpha(k)})} a_k u^{-s_k}.
\end{equation}

It is well-known that strong (weak) topology and strong (weak)
operator topology coincide on every bounded subset of $\mk{A}$. 
Therefore, there is a direct correspondence between strong (weak) convergence of the polynomial
multiple ergodic averages (\ref{eq:mult-erg-ave}) and strong (weak)
\emph{operator} convergence of the polynomial entangled ergodic averages
(\ref{eq:ent-ave-for-mult}).

\begin{A}\label{prop:entangled-conv-equiv}
Let $(\mk{A}, \ph, \beta)$ be a von Neumann dynamical system and $H$ and $u$ as
above. Let further $a_1,\ldots,a_k\in \mk{A}$. Then the multiple ergodic averages
(\ref{eq:mult-erg-ave}) converge strongly (weakly) if and only if the
entangled averages (\ref{eq:ent-ave-for-mult}) converge  
in the strong (weak) operator topology. 
\end{A}

We now show that under certain compactness asssumptions, the averages  (\ref{eq:ent-ave-for-mult}) converge in the strong operator topology. This is a generalization of the results in Eisner, K-K \cite{EKK}

\begin{A}
Let $H$ be a Hilbert space, $U\in\mc{L}(H)$ a unitary operator, $p_1,\ldots,p_r\in\mc{P}$ and $\alpha:\{1,\ldots,k\}\rightarrow\{1,\ldots,r\}$ a surjective mapping. Let further $A_1,\ldots,A_k\in \mc{L}(H)$ be such that $\{A_k U^{-n}: n\in \mb{N}^+\}$ and $\{A_j U^n: n\in \mb{N}^+\}$ are relatively compact in $\mc{L}(H)$ for the strong operator topology for every $1\leq j\leq k-1$. Then the polynomial entangled ergodic averages
 \begin{equation}\label{eq:ent-ave-for-mult-poly}
\frac{1}{N^r}\sum_{n_1,\ldots,n_r=1}^{N} U^{p_{\alpha(1)}(n_{\alpha(1)})}A_1U^{p_{\alpha(2)}(n_{\alpha(2)})}A_2\cdots U^{p_{\alpha(k)}(n_{\alpha(k)})} A_k U^{-\sum_{j=1}^k p_{\alpha(j)}(n_{\alpha(j)})}
\end{equation}
converge in the strong operator topology.
\end{A}
\begin{proof}
The proof is based on induction, and is in essence a polynomial version of that of Theorem 3 in \cite{EKK}, and for detailed arguments we refer to the proof given there.
The polynomial versions of the required lemmas have been proven for the Hilbert space case in Section \ref{Sect:awps}.
The only significant difference is that an extra step is needed here to set up the induction, as the last power in the averages considered is a sum of polynomials rather than a single polynomial.

Since $U$ is unitary, it induces a Jacobs-deLeeuw-Glicksberg decomposition of $H$ into the orthogonal sum $H_\mr{s}\oplus H_\mr{r}$, cf. Proposition \ref{prop:JGdL_Hilbert}. By linearity it is enough to show that the averages applied to any $x\in H_\mr{s}$ and $x\in H_\mr{r}$ converge.

Let us first assume that $x\in H_\mr{s}$. We wish to show that
\[
\frac{1}{N^r}\sum_{n_1,\ldots,n_r=1}^{N} U^{p_{\alpha(1)}(n_{\alpha(1)})}A_1U^{p_{\alpha(2)}(n_{\alpha(2)})}A_2\cdots U^{p_{\alpha(k)}(n_{\alpha(k)})} A_k U^{-\sum_{j=1}^k p_{\alpha(j)}(n_{\alpha(j)})}x
\]
converges to $0$ in norm. To this end note that by assumption, $L:=\{A_k U^{-n}x: n\in \mb{N}^+\}$ is relatively norm-compact. Denote its closure by $K$.
We shall need that the dual space of the smallest $U$-invariant subspace $Y$
containing $K$ is separable. Indeed, as $Y$ is a Hilbert space, this is equivalent to $Y$ itself being separable, which follows from the countability of the generating set $L$.
Therefore one may by Remark \ref{R:separable} find a sequence $\left(t_j\right)_{j\in\mb{N}^+}$ of density 1 such that $\lim_{j\to\infty} A_k U^{-\left|\{\alpha^{-1}(1)\}\right|\cdot p_{1}(t_j)}y =0$ for any $y\in Y$. By compactness this convergence is actually uniform on $K$. Since $ U^{p_{\alpha(1)}(n_{\alpha(1)})}A_1U^{p_{\alpha(2)}(n_{\alpha(2)})}A_2\cdots U^{p_{\alpha(k)}(n_{\alpha(k)})}$ is uniformly bounded and $ A_k U^{-\sum_{j=1}^k p_{\alpha(j)}(n_{\alpha(j)})}x$ can be rewritten as
\[
y_{n_1}:= A_k U^{-\left|\alpha^{-1}(1)\right|\cdot p_{1}(n_1)}\left(U^{-\sum_{j=2}^r \left|\alpha^{-1}(j)\right|\cdot  p_{j}(n_{j})}x\right),
\]
the norm convergence of the means follows from Lemma \ref{lemma:K-vN}.

Let now $x\in H_\mr{r}$. By uniform boundedness of the operator products involved, one may by the standard density argument assume that $x$ is an eigenvector to some unimodular eigenvalue $\lambda\in\Gamma$. Then
\begin{eqnarray*}
&&\frac{1}{N^r}\sum_{n_1,\ldots,n_r=1}^{N} U^{p_{\alpha(1)}(n_{\alpha(1)})}A_1U^{p_{\alpha(2)}(n_{\alpha(2)})}A_2\cdots U^{p_{\alpha(k)}(n_{\alpha(k)})} A_k U^{-\sum_{j=1}^k p_{\alpha(j)}(n_{\alpha(j)})}x\\
&=&\frac{1}{N^r}\sum_{n_1,\ldots,n_r=1}^{N} U^{p_{\alpha(1)}(n_{\alpha(1)})}A_1U^{p_{\alpha(2)}(n_{\alpha(2)})}A_2\cdots U^{p_{\alpha(k)}(n_{\alpha(k)})} A_k \lambda^{-\sum_{j=1}^k p_{\alpha(j)}(n_{\alpha(j)})}x\\
&=&\frac{1}{N^r}\sum_{n_1,\ldots,n_r=1}^{N} \left(\ol{\lambda}U\right)^{p_{\alpha(1)}(n_{\alpha(1)})}A_1 \left(\ol{\lambda}U\right)^{p_{\alpha(2)}(n_{\alpha(2)})}A_2\cdots \left(\ol{\lambda}U\right)^{p_{\alpha(k)}(n_{\alpha(k)})} \left(A_k x\right).
\end{eqnarray*}

This is now a form where each power is a single polynomial, and the induction arguments from the proof of Theorem 3 in \cite{EKK} can be applied to show convergence.
\end{proof}

We can thus conclude the following for polynomial dynamical systems.

\begin{Cor}
Let $(\mk{A}, \ph, \beta)$ be a von Neumann dynamical system, with unitary representation $u\cdot u^{-1}$ of $\beta$ on the GNS-space $H$ pertaining to $\ph$.
Let further $r,k\in\mb{N}^+$ with $r\leq k$, $\alpha:\{1,\ldots,k\}\rightarrow\{1,\ldots,r\}$ be a surjective mapping, $p_1,p_2,\ldots,p_r\in\mc{P}$ and let  $s_l:=\sum_{d=1}^l p_{\alpha(d)}(n_{\alpha(d)})$ for each $1\leq l\leq k$.
Assume now that $a_1,\ldots,a_k\in \mk{A}$ are such that $\{a_k u^{-n}: n\in \mb{N}^+\}$ and $\{a_j u^n: n\in \mb{N}^+\}$ are relatively compact in $\mc{L}(H)$ for the strong operator topology for every $1\leq j\leq k-1$. Then the polynomial multiple averages
\[
\frac{1}{N^r} \sum_{n_1,\ldots,n_r=1}^N \beta^{s_1}(a_1)\beta^{s_2}(a_2)\cdots\beta^{s_k}(a_k)
\]
converge strongly.
\end{Cor}

\bibliographystyle{amsplain}

\end{document}